\documentclass{article}
\usepackage{amssymb,amsmath,bm,geometry,graphics,graphicx,url,color}

\def\sinc{\mathrm{sinc}}

\def\pdt2{\partial_t^2}
\def\pdx2{\partial_x^2}
\newcommand{\norm}[1]{\left\Vert#1\right\Vert}
\newcommand{\normmm}[1]{{\left\vert\kern-0.25ex\left\vert\kern-0.25ex\left\vert #1
    \right\vert\kern-0.25ex\right\vert\kern-0.25ex\right\vert}}

\newcommand{\abs}[1]{\left\vert#1\right\vert}

\newtheorem{theo}{Theorem}[section]

\newtheorem{rem}[theo]{Remark}
\newtheorem{defi}[theo]{Definition}

\newtheorem{assum}[theo]{Assumption}
\newtheorem{prop}[theo]{Proposition}

\def\no{\noindent}

\title{Long-time oscillatory energy conservation of total energy-preserving methods for highly oscillatory\\ Hamiltonian systems
\thanks{The first author was supported in part by the Alexander von Humboldt Foundation and by
the Natural Science Foundation of Shandong Province (Outstanding
Youth Foundation) under Grant ZR2017JL003. The second author was
supported in part by the National Natural Science Foundation of
China under Grant 11671200.}}

\author{Bin Wang\,\footnote{School of Mathematical Sciences, Qufu Normal
University, Qufu  273165,  P.R.China; Mathematisches Institut,
University of T\"{u}bingen, Auf der Morgenstelle 10, 72076
T\"{u}bingen, Germany. E-mail:~{\tt wang@na.uni-tuebingen.de}} \and
Xinyuan Wu\thanks{School of Mathematical Sciences, Qufu Normal
University, Qufu  273165,  P.R.China; Department of Mathematics,
Nanjing University,  Nanjing 210093, P.R. China. E-mail:~{\tt
xywu@nju.edu.cn}} }

\begin{document}
\maketitle

\begin{abstract}
For an integrator when applied to a highly oscillatory system, the
near conservation of the oscillatory energy over long times is an
important aspect. In this paper, we study the long-time near
conservation of oscillatory energy for the adopted average vector
field (AAVF) method when applied to highly oscillatory Hamiltonian
systems. This AAVF method is an extension of the average vector
field method  and preserves the total energy of highly oscillatory
Hamiltonian systems exactly. This paper is devoted to analysing
another important property of AAVF method,  i.e., the near
conservation of its oscillatory energy in a long term. The long-time
oscillatory energy conservation is obtained via constructing a
modulated Fourier expansion of the AAVF method and deriving an
almost invariant of the expansion. A similar result of the method in
the multi-frequency case is also presented in this paper.
\medskip

\no{Keywords:} highly oscillatory Hamiltonian systems,  modulated
Fourier expansion,  AAVF  method, energy-preserving methods,
long-time oscillatory energy conservation.

\medskip
\no{MSC (2000):} 65P10, 65L05

\end{abstract}

\section{Introduction}\label{intro}

This paper is concerned with the long-time oscillatory energy
behaviour  of energy-preserving methods   for the highly oscillatory
Hamiltonian system
\begin{equation}
 \left\{\begin{aligned}
 &\dot{q}=\nabla_p H(q,p),\qquad \ \ q(0)=q_{0},\\
 &\dot{p}=-\nabla_q H(q,p),\qquad  p(0)=p_{0}, \end{aligned}\right.
\label{H-s}%
\end{equation}
where the Hamiltonian function is given by
\begin{equation}H(q,p)=\frac{1}{2}\big(\norm{p}^2+\norm{\Omega q}^2\big)+U(q).\label{H}%
\end{equation}
Here $U(q)$ is a real-valued function. According to the partition of
the square matrix
$$\Omega= \left(
                          \begin{array}{cc}
                            0_{d_1\times d_1}  & 0 _{d_2\times d_2}  \\
                            0_{d_1\times d_1} & \omega I_{d_2\times d_2} \\
                          \end{array}
                        \right)
$$ with a large positive parameter $\omega$,   the vectors $p = (p_1,
p_2) \in \mathbb{R}^{d_1}\times \mathbb{R}^{d_2}$ and $q = (q_1,
q_2)\in \mathbb{R}^{d_1}\times \mathbb{R}^{d_2}$ are partitioned
accordingly.  As is known, the oscillatory energy of the system
\eqref{H-s} is
\begin{equation}
I(q,p)=\frac{1}{2}p_2^\intercal p_2 +\frac{1}{2}\omega^2q_2^\intercal q_2 \label{oscillatory energy}%
\end{equation}
and it  is nearly conserved over long times along the solution of
\eqref{H-s} (see \cite{hairer2006}).  Our attention of this paper
will particularly focus on the near conservation of the oscillatory
energy \eqref{oscillatory energy} for energy-preserving methods over
long-time intervals.

In order to preserve the  total energy  of Hamiltonian systems
exactly by numerical methods, energy-preserving (EP) methods have
been proposed and researched. In the recent decades, various kinds
of EP methods have been derived, such as the average vector field
(AVF) method \cite{Celledoni2010,Celledoni14,Quispel08}, discrete
gradient  methods \cite{McLachlan14,McLachlan99},  the
energy-preserving collocation methods \cite{Cohen-2011,Hairer2010},
Hamiltonian Boundary Value Methods
\cite{Brugnano2010,Brugnano2012b}, energy-preserving
exponentially-fitted   methods \cite{Miyatake2014,Miyatake2015}, and
time finite elements  methods
\cite{Betsch00,Li_Wu(na2016),wang2018-JCP}. By taking advantage the
frequency matrix of  second-order highly oscillatory systems, a
novel adopted AVF (AAVF) method has been formulated and studied in
\cite{wang2012-PLA,wu2013-JCP} for the highly oscillatory
Hamiltonian system \eqref{H}.  It has been proved in
\cite{wang2012-PLA,wu2013-JCP} that this AAVF method exactly
preserves the  total energy  \eqref{H} and it  reduces to the AVF
method when the frequency matrix vanishes. However, most existing
publications dealing with EP methods focus on the formulation of the
methods and the analysis of the EP property. It seems that the
long-time behaviour of AAVF method concerning other
structure-preserving aspects has never been studied in the
literature, such as the long-time numerical conservation of
oscillatory energies. \emph{As is well known that an important
property of highly oscillatory systems is the near conservation of
the oscillatory energy over long times}.

On the other hand, in the recent two decades, modulated Fourier
expansion has been presented and developed   as an important
mathematical tool in the study of the long-time behaviour for
numerical methods/differential equations (see, e.g.
\cite{Cohen15,Gauckler13,Hairer16}). It was firstly given  in
\cite{Hairer00} and has been used in the long-time analysis for
various numerical methods, such as  for the St\"{o}rmer--Verlet
method in \cite{Hairer00-1},  for trigonometric integrators in
\cite{Cohen15,hairer2006}, for an implicit-explicit method in
\cite{McLachlan14siam,Stern09}, for heterogeneous multiscale methods
in \cite{Sanz-Serna09} and for splitting methods in
\cite{Gauckler17,Gauckler10-1}. However, it is noted that, until
now,   the technique of modulated Fourier expansions has not been
well applied to the long-term analysis for any energy-preserving
method in the literature.

Based on the facts stated above,  the  main contribution of this
paper is to analyse the long-time oscillatory energy conservation
for the AAVF method. To this end, the technique of modulated Fourier
expansions with some adaptations will be used in the analysis. To
our knowledge, this paper is the first one that rigorously studies
the remarkable long-time oscillatory energy conservation   of EP
methods on highly oscillatory Hamiltonian systems by using modulated
Fourier expansions.

The rest of this paper is organised as follows. We first present the
scheme of AAVF method and carry out an illustrative numerical
experiment in Section \ref{sec:Formulation}. Section \ref{sec:mfe of
the methods} derives the modulated Fourier expansion of the AAVF
method and analyse the bounds of the modulated Fourier functions. In
Section \ref{sec:oscillatory energy conservation}, we show an almost
invariant of the modulation system and then the main result
concerning the long-time  oscillatory energy conservation of AAVF
method is derived. Section \ref{sec: multi-frequency case} extends
the analysis  to multi-frequency case and studies the long-time
 conservation of AAVF method when applied to
multi-frequency highly oscillatory Hamiltonian systems. The last
section includes the concluding remarks of this paper.


\section{The AAVF method and illustrative numerical experiments}\label{sec:Formulation}
The highly oscillatory Hamiltonian system \eqref{H} can be rewritten
as a system of  second-order differential equations
\begin{equation}
q^{\prime\prime}(t)+\Omega^2q(t)=f(q(t)),  \qquad
q(0)=q_0,\ \ q'(0)=p_0,\label{prob}%
\end{equation}
where $f$ is the negative gradient of the  real-valued function
$U(q)$. For effectively integrating this second-order highly
oscillatory system, a novel kind of EP methods was  derived in
\cite{wang2012-PLA,wu2013-JCP}.

\begin{defi}
\label{defAAVF}  (See \cite{wang2012-PLA,wu2013-JCP}) The   adapted AVF (AAVF)      method for solving \eqref{prob} is defined by%
\begin{equation}\left\{
\begin{aligned}
 q_{n+1}&=\phi_0(V)q_n+h\phi_1(V)p_n+h^2\phi_2(V)\displaystyle\int_{0}^{1}f((1-\tau)q_{n}+\tau q_{n+1})d\tau,\\
p_{n+1}&=-h\Omega^2\phi_1(V)q_n+\phi_0(V)p_n+h\phi_1(V)\displaystyle\int_{0}^{1}f((1-\tau)q_{n}+\tau
q_{n+1})d\tau,
\end{aligned}\right.\label{AAVFmethod}%
\end{equation}
where $h$ is the stepsize, and $\phi_0, \phi_1$ and $\phi_2$ are
matrix-valued functions of $V=h^{2}\Omega^2$ defined by
\begin{equation}
\phi_{l}(V):=\sum\limits_{k=0}^{\infty}\dfrac{(-1)^{k}V^{k}}{(2k+l)!},\
\ l=0,1,2. \label{Phi01}
\end{equation}
\end{defi}

 It is noted that
in terms of   this definition, we have
\begin{equation*}
\phi_{0}(V)=\cos(h\Omega),\qquad  \phi_{1}(V) = \sin(h\Omega) (h\Omega)^{-1},\qquad  \phi_{2}(V) = (I-\cos(h\Omega)) (h\Omega)^{-2}.
\end{equation*}
It  can be observed that this method \eqref{AAVFmethod}  reduces to
the AVF method when $\Omega=0$. It also has been shown in
\cite{wang2012-PLA,wu2013-JCP} that this method is symmetric and
exactly preserves the total energy \eqref{H}. In this paper, we pay
attention to its long-time numerical behaviour in oscillatory energy
preservation and prove the result by modulated Fourier expansions.

 As an illustrative numerical
example, we apply this method  to the   the Fermi--Pasta--Ulam
problem, which can be expressed by a Hamiltonian system with the
Hamiltonian
 \begin{equation*}
\begin{array}
[c]{ll}%
H(y,x) &
=\dfrac{1}{2}\textstyle\sum\limits_{i=1}^{2m}y_{i}^{2}+\dfrac
{\omega^{2}}{2}\textstyle\sum\limits_{i=1}^{m}x_{m+i}^{2}+\dfrac{1}{4}%
[(x_{1}-x_{m+1})^{4}\\
& +\textstyle\sum\limits_{i=1}^{m-1}(x_{i+1}-x_{m+i-1}-x_{i}-x_{m+i}%
)^{4}+(x_{m}+x_{2m})^{4}].
\end{array}
\end{equation*}
For the AAVF formula \eqref{AAVFmethod}, we consider    applying
 midpoint rule, Simpson's rule and four-point Gauss-Legendre's rule   to the integral
and denote the corresponding methods by AAVF1, AAVF2 and AAVF3,
respectively. Following \cite{hairer2006}, we choose $m=3$ and
\begin{equation*}\label{initial date} \ x_{1}(0)=1,\ y_{1}(0)=1,\
x_{4}(0)=\dfrac{1}{\omega},\ y_{4}(0)=1
 \end{equation*}
with zero for the remaining initial values.  The system is
integrated in the interval $[0,1000]$ with   $h=0.02,0.01$ and
$\omega=200$. We remark that the values of $h\omega$ are $4$ and
$2$. The errors of the oscillatory energy $I$ against $t$ for
different  methods are shown in Figs. \ref{p0}-\ref{p2}.  From the
results, it can be observed a fact that   these three methods
approximately conserve  the oscillatory  energy $I$ very well over a
long term. Moreover, it seems that no matter  which quadrature is
used, there is no difference in the oscillatory energy conservation.
All the phenomena will be explained theoretically in the rest of
this paper.

 \begin{figure}[ptb]
\centering
\includegraphics[width=6cm,height=3cm]{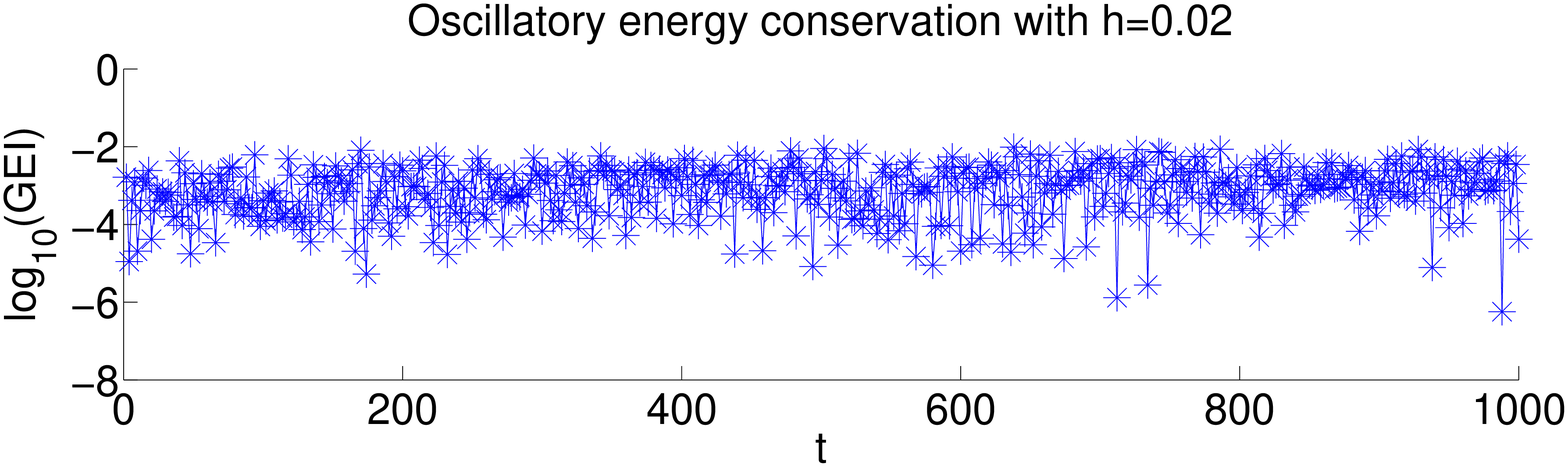}
\includegraphics[width=6cm,height=3cm]{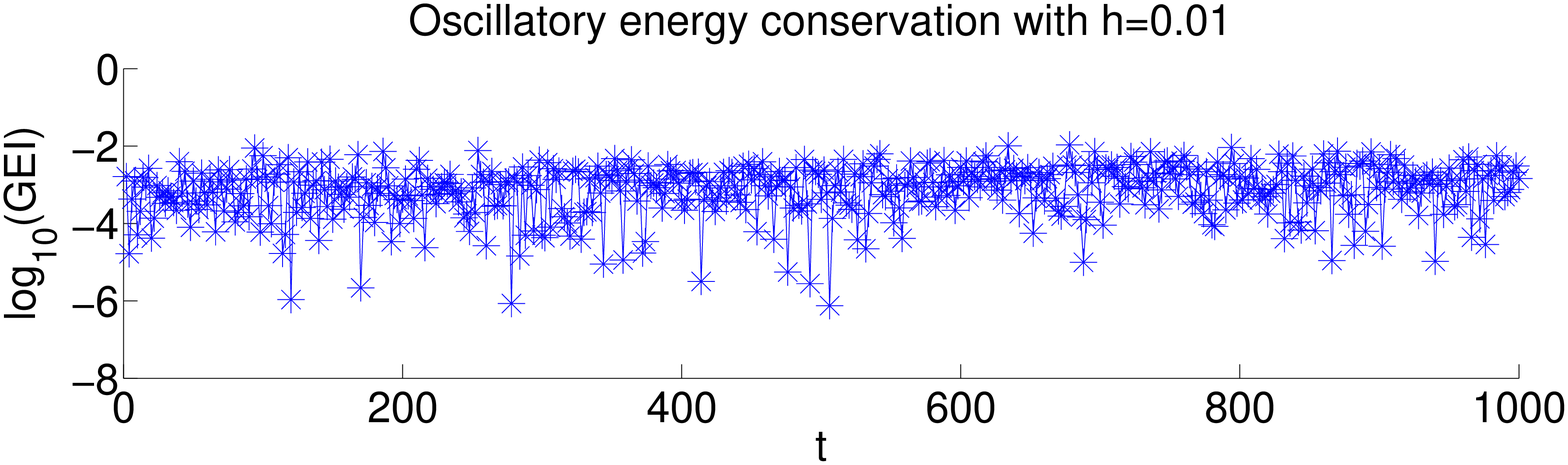}
\caption{AAVF 1: the logarithm of the oscillatory energy errors
against $t$.} \label{p0}
\end{figure}

 \begin{figure}[ptb]
\centering
\includegraphics[width=6cm,height=3cm]{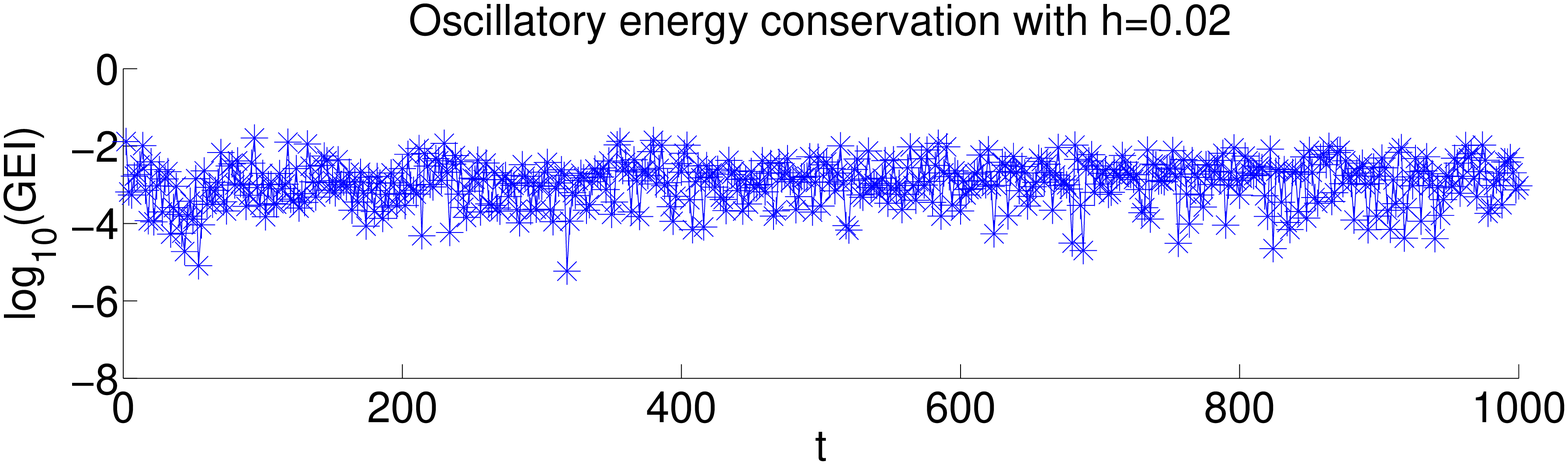}
\includegraphics[width=6cm,height=3cm]{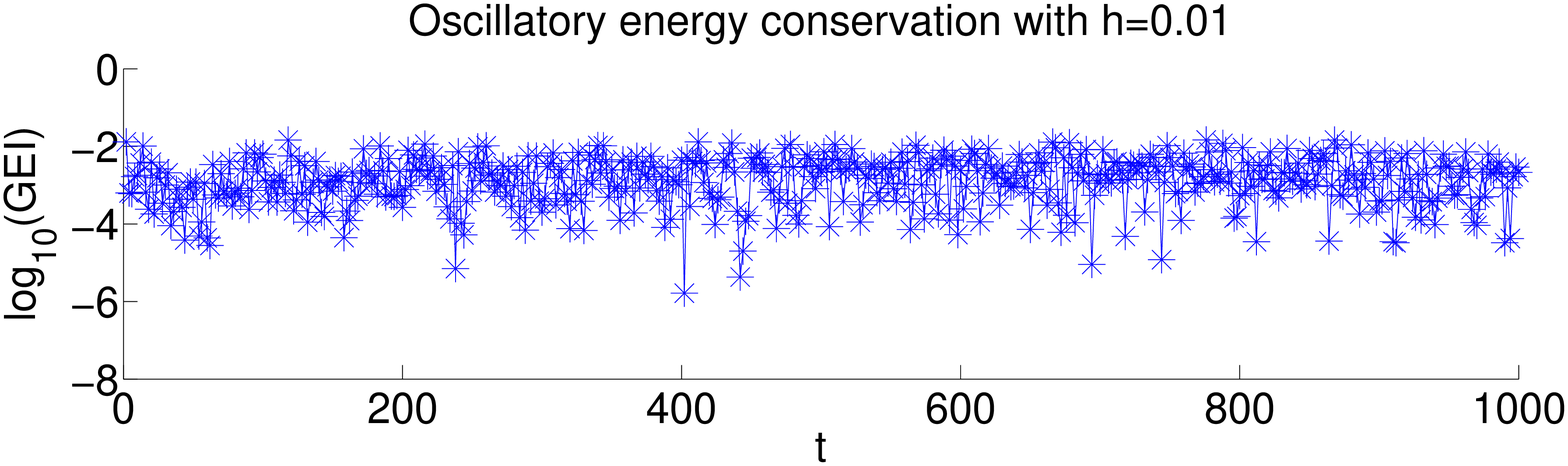}
\caption{AAVF 2: the logarithm of the oscillatory energy errors
against $t$.} \label{p1}
\end{figure}

 \begin{figure}[ptb]
\centering
\includegraphics[width=6cm,height=3cm]{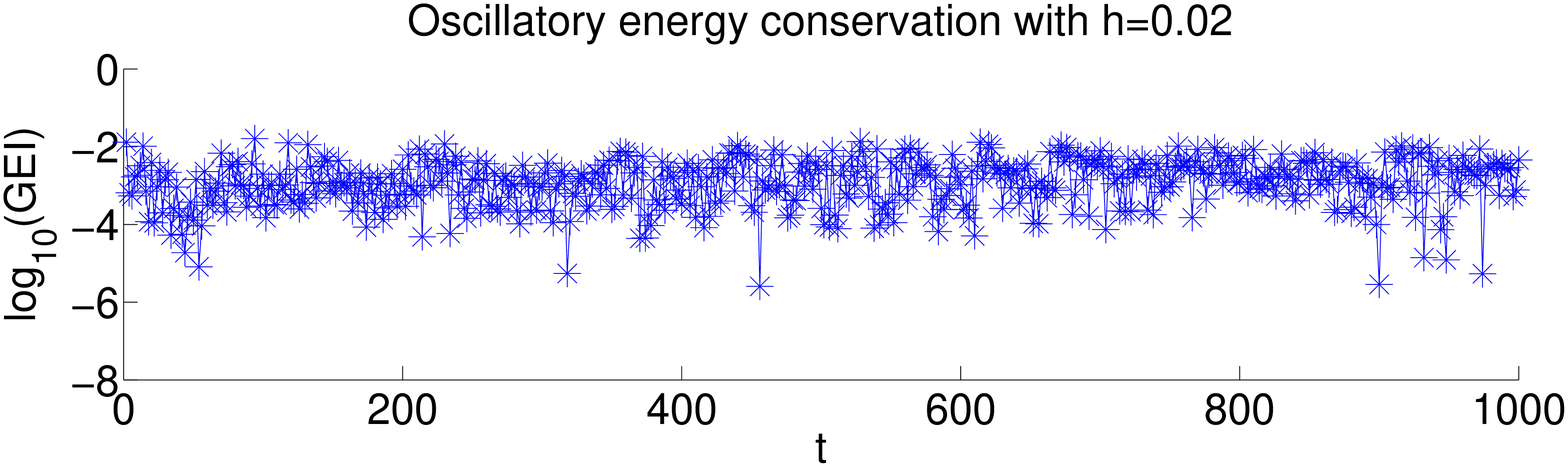}
\includegraphics[width=6cm,height=3cm]{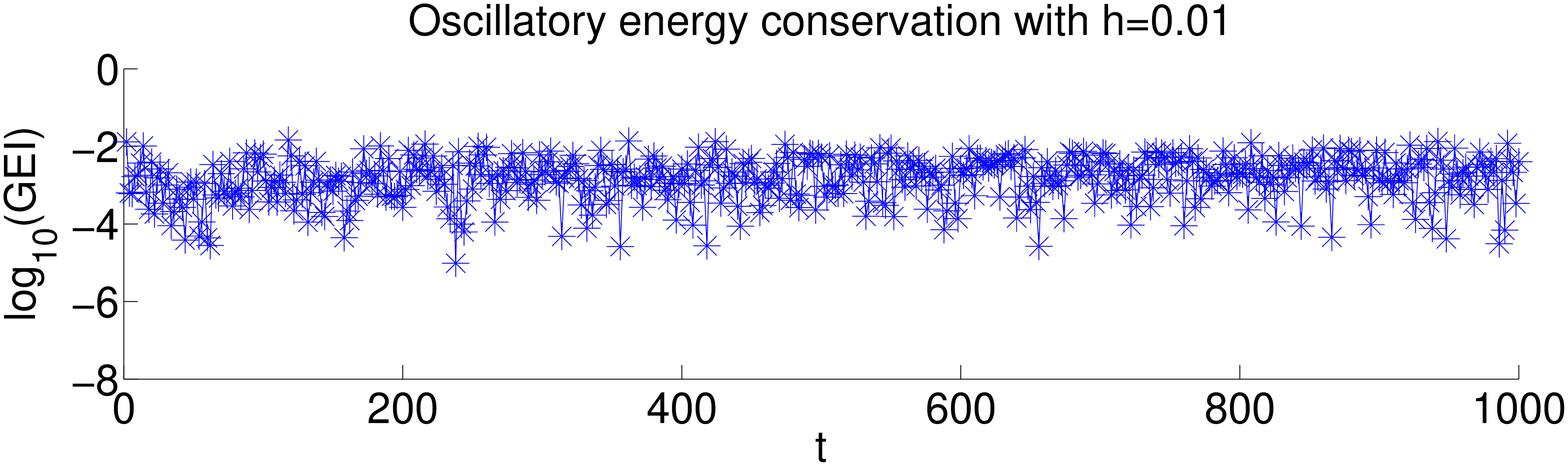}
\caption{AAVF 3: the logarithm of the oscillatory energy errors
against $t$.} \label{p2}
\end{figure}

\section{Modulated Fourier expansion} \label{sec:mfe of the methods}
In this section, we derive a modulated Fourier expansion of the AAVF
method.  The following assumptions are needed in our analysis.

\begin{assum}\label{ass}
\begin{itemize}

\item  The
initial values of \eqref{H-s} are assumed to satisfy
\begin{equation}\label{energy bound}\frac{1}{2}
\norm{p(0)}^2+\frac{1}{2} \norm{\Omega q(0)}^2\leq E\end{equation}
 with a   constant $E$  independent of $\omega$.

\item The numerical solution of the AAVF method is assumed to stay in a compact set.

\item The  stepsize  is  required to have a lower bound  such that
$ h\omega \geq c_0 > 0. $
\item We assume that the numerical non-resonance condition is true
\begin{equation}
\abs{\sin\big(\frac{1}{2}kh\omega\big)} \geq c \sqrt{h}\ \ \mathrm{for} \ \ k=1,2,\ldots,N\ \   \mathrm{with} \ \ N\geq2.\label{numerical non-resonance cond}%
\end{equation}
\end{itemize}
\end{assum}

These assumptions have been considered many times in the long-term
analysis of other methods without EP property and we refer to
\cite{Cohen05,Hairer00,hairer2006} for example.

  In this paper, we   define   five
operators   by
\begin{equation}\label{LLL}
\begin{aligned}L_1(hD):&=\mathrm{e}^{hD}-2\cos(h\Omega)+\mathrm{e}^{-hD},\\
L_2(hD):&=\mathrm{e}^{\frac{1}{2}hD}+\mathrm{e}^{-\frac{1}{2}hD},\\
L_3(hD):&=(\mathrm{e}^{hD}-1)(\mathrm{e}^{hD}+1)^{-1},\\
L_4(hD,\tau,k):&= (1-\tau)\mathrm{e}^{-\mathrm{i} \frac{h}{2}k\omega
}\mathrm{e}^{-
 \frac{h}{2}D} +\tau\mathrm{e}^{\mathrm{i} \frac{h}{2}k\omega} \mathrm{e}^{
 \frac{h}{2}D},\\
 L(hD):&=(L_2^{-1} L_1)(hD),\end{aligned}
\end{equation} where $D$ is the
differential operator. The  following properties of these operators
will be used in our analysis.
\begin{prop}\label{lhd pro}
 The Taylor expansions of $  L(hD)$   are given by
\begin{equation*}
\begin{aligned}&L(hD)= \left(
                           \begin{array}{cc}
                             0 & 0 \\
                             0 & 1-\cos(h\omega) \\
                           \end{array}
                         \right)-\frac{1}{8}\left(
                           \begin{array}{cc}
                             4 & 0 \\
                             0 & 3+\cos(h\omega) \\
                           \end{array}
                         \right)
(\textmd{i} hD)^2+\cdots,\\
 &L(hD+\mathrm{i}h\omega)=\left(
                           \begin{array}{cc}
                             -4\csc(h\omega) \sin^3(\frac{h\omega}{2}) & 0 \\
                             0 & 0\\
                           \end{array}
                         \right)\\
                         &+\left(
                           \begin{array}{cc}
                             \frac{3+\cos(h\omega)}{2} \sec(\frac{h\omega}{2})\tan(\frac{h\omega}{2})& 0 \\
                             0 & 2\sin(\frac{h\omega}{2}) \\
                           \end{array}
                         \right)
(\textmd{i} hD)+\cdots,\\
&L(hD+\mathrm{i}kh\omega)=\left(
                           \begin{array}{cc}
                             -4\csc(kh\omega) \sin^3(\frac{kh\omega}{2}) & 0 \\
                             0 & (\cos(kh\omega)-\cos(h\omega))\sec(\frac{kh\omega}{2})\\
                           \end{array}
                         \right)\\
                         &+\left(
                           \begin{array}{cc}
                             \frac{3+\cos(kh\omega)}{2} \sec(\frac{kh\omega}{2})\tan(\frac{kh\omega}{2})& 0 \\
                             0 & \frac{2+\cos(kh\omega)+\cos(h\omega)}{2} \sec(\frac{kh\omega}{2}) \tan(\frac{kh\omega}{2}) \\
                           \end{array}
                         \right)
(\textmd{i} hD)+\cdots
\end{aligned}
\end{equation*}
for $\abs{k}>1$. The operator $L_3(hD)$ can be   expressed in its
Taylor expansions as follows:
\begin{equation*}
\begin{aligned}&L_3(hD)= \frac{1}{2}(hD)-\frac{1}{24}(hD)^3+\cdots,\\
&L_3(hD+\mathrm{i}kh\omega)=\tan(\frac{kh\omega}{2})\textmd{i}
+\frac{1}{1+\cos(kh\omega)}(hD)+\cdots
\end{aligned}
\end{equation*}
for $\abs{k}>0$. Moreover, for the operator $L_4(hD,\tau,k)$ with
$\abs{k}>0$, the following result holds
\begin{equation*}
\begin{aligned}
L_4(hD,
\frac{1}{2},k)=\cos(\frac{kh\omega}{2})+\frac{1}{2}\sin(\frac{kh\omega}{2})(\textmd{i}
hD)+\cdots.
\end{aligned}
\end{equation*}

\end{prop}

\begin{theo}\label{energy thm}
Suppose that the conditions given in Assumption \ref{ass} are true.
 The numerical solution of the AAVF
method \eqref{AAVFmethod} admits the following modulated Fourier
expansion for $0 \leq t=nh \leq T$:
\begin{equation}
\begin{aligned} &q_{n}=  \sum\limits_{|k|<N} \mathrm{e}^{\mathrm{i}k\omega t}\zeta_h^k(t)+R_{h,N}(t),\ \
 p_{n}= \sum\limits_{|k|<N} \mathrm{e}^{\mathrm{i}k\omega t}\eta_h^k(t)+S_{h,N}(t),\\
\end{aligned}
\label{MFE-ERKN}%
\end{equation}
where $N$ is a fixed integer determined by \eqref{numerical
non-resonance cond} and the remainder terms are bounded by
\begin{equation}
 R_{h,N}(t)=\mathcal{O}(h^{N}),\ \ \ \  S_{h,N}(t)=\mathcal{O}(h^{N-1}).\\
\label{remainder}%
\end{equation}
The coefficient functions  $ \zeta_h^k, \eta_h^k$ as well as all
their derivatives are bounded by
\begin{equation}
\begin{aligned} &\zeta^0_{h,1}=\mathcal{O}(1),\qquad \quad    \eta^0_{h,1}=\mathcal{O}(1),\ \  \qquad
\zeta^0_{h,2}=\mathcal{O}(h^2),\quad \ \ \ \ \ \eta^0_{h,2}=\mathcal{O}(h^{ \frac{3}{2}}),  \\
&\zeta^1_{h,1}=\mathcal{O}(h^2),\  \ \ \ \ \ \
\eta^1_{h,1}=\mathcal{O}(h),\quad \ \ \ \ \
 \zeta^1_{h,2}=\mathcal{O}(h),\quad\ \ \ \    \ \ \eta^1_{h,2}=\textmd{i} \omega\zeta^1_{h,2}+\mathcal{O}(h),  \\
&\zeta^k_{h,1}=\mathcal{O}(h^{k+1}),\quad \
\eta^k_{h,1}=\mathcal{O}(h^{k}),\quad \ \ \ \
\zeta^k_{h,2}=\mathcal{O}(h^{k+1}),\quad \ \
\eta^k_{h,2}=\mathcal{O}(h^k)   \\
\end{aligned}
\label{coefficient func}%
\end{equation}
for $k=2,\ldots,N-1$. Moreover, we have
$\zeta^{-k}=\overline{\zeta^{k}}$ and
$\eta^{-k}=\overline{\eta^{k}}$. The constants symbolised by the
notation depend on the constants  from Assumption \ref{ass} and the
final time $T$,  but are independent of $h$ and $\omega$.
\end{theo}
\begin{proof} In this proof, we will construct the functions
\begin{equation}
\begin{aligned} &q_{h}(t)= \sum\limits_{|k|<N} \mathrm{e}^{\mathrm{i}k\omega
t}\zeta_h^k(t),\quad
 \ p_{h}(t)= \sum\limits_{|k|<N} \mathrm{e}^{\mathrm{i}k\omega t}\eta_h^k(t)\\
\end{aligned}
\label{MFE-1}%
\end{equation}
with smooth  coefficient functions $ \zeta_h^k$ and $\eta_h^k$, such
that there is only a small defect when \eqref{MFE-1} is inserted
into the numerical scheme \eqref{AAVFmethod}.

 \vskip1mm \textbf{I. Construction of the
coefficients functions.}

It follows from the symmetry of the AAVF
 method that
\begin{equation}
\begin{aligned}&q_{n+1}-2\cos(h\Omega)q_{n}+q_{n-1}\\
=&h^2\phi_2(V)\big[\displaystyle\int_{0}^{1}f((1-\tau)q_{n}+\tau
q_{n+1})d\tau+
\displaystyle\int_{0}^{1}f((1-\tau)q_{n}+\tau q_{n-1})d\tau\big]\\
=&h^2\phi_2(V)\big[\displaystyle\int_{0}^{1}f((1-\tau)q_{n}+\tau
q_{n+1})d\tau+ \displaystyle\int_{0}^{1}f((1-\tau)q_{n-1}+\tau
q_{n})d\tau\big],
\end{aligned}\label{MFE-2}%
\end{equation}
where we have used the following property
$$\displaystyle\int_{0}^{1}f((1-\tau)q_{n}+\tau
q_{n-1})d\tau=\displaystyle\int_{0}^{1}f((1-\tau)q_{n-1}+\tau
q_{n})d\tau.$$
 For the   term $(1-\tau)q_{n}+\tau
q_{n+1}$, we look for a  function of the form
\begin{equation*}
\begin{aligned} &\tilde{q}_{h}(t+\frac{h}{2},\tau)=\sum\limits_{|k|<N}
\mathrm{e}^{\mathrm{i}k\omega
(t+\frac{h}{2})}\xi_h^k(t+\frac{h}{2},\tau)
\end{aligned}
\end{equation*}
as its modulated Fourier expansion. Then one has
\begin{equation*}
\begin{aligned} \tilde{q}_{h}(t+\frac{h}{2},\tau)&=(1-\tau)\sum\limits_{|k|<N} \mathrm{e}^{\mathrm{i}k\omega
t}\zeta_h^k(t)+\tau \sum\limits_{|k|<N}
\mathrm{e}^{\mathrm{i}k\omega (t+h)}\zeta_h^k(t+h)\\
&=\sum\limits_{|k|<N} \mathrm{e}^{\mathrm{i}k\omega
(t+\frac{h}{2})}\Big((1-\tau)\mathrm{e}^{-\mathrm{i}k\omega
 \frac{h}{2}}\mathrm{e}^{-
 \frac{h}{2}D} +\tau\mathrm{e}^{\mathrm{i}k\omega
 \frac{h}{2}} \mathrm{e}^{
 \frac{h}{2}D}\Big)\zeta_h^k(t+\frac{h}{2}),
\end{aligned}
\end{equation*}
which yields
\begin{equation}\label{xip}
\begin{aligned} \xi_h^k(t+\frac{h}{2},\tau)=&\Big((1-\tau)\mathrm{e}^{-\mathrm{i}k\omega
 \frac{h}{2}}\mathrm{e}^{-
 \frac{h}{2}D} +\tau\mathrm{e}^{\mathrm{i}k\omega
 \frac{h}{2}} \mathrm{e}^{
 \frac{h}{2}D}\Big)\zeta_h^k(t+\frac{h}{2})\\
 =&L_4(hD,\tau,k)\zeta_h^k(t+\frac{h}{2}).
\end{aligned}
\end{equation}
Similarly, for $(1-\tau)q_{n-1}+\tau q_{n}$, we have the following
modulated Fourier expansion
\begin{equation*}
\begin{aligned} &\tilde{q}_{h}(t-\frac{h}{2},\tau)=\sum\limits_{|k|<N}
\mathrm{e}^{\mathrm{i}k\omega
(t-\frac{h}{2})}\xi_h^k(t-\frac{h}{2},\tau)
\end{aligned}
\end{equation*}
with
\begin{equation}\label{xim}
\begin{aligned} \xi_h^k(t-\frac{h}{2},\tau)=\Big((1-\tau)\mathrm{e}^{-\mathrm{i}k\omega
 \frac{h}{2}}\mathrm{e}^{-
 \frac{h}{2}D} +\tau\mathrm{e}^{\mathrm{i}k\omega
 \frac{h}{2}} \mathrm{e}^{
 \frac{h}{2}D}\Big)\zeta_h^k(t-\frac{h}{2})=L_4(hD,\tau,k)\zeta_h^k(t-\frac{h}{2}).
\end{aligned}
\end{equation}
Inserting these modulated Fourier expansions into \eqref{MFE-2}
implies
\begin{equation*}
\begin{aligned}&q_{h}(t+h)-2\cos(h\Omega)q_{h}(t)+q_{h}(t-h)\\
=&h^2\phi_2(V)\big[\displaystyle\int_{0}^{1}f(\tilde{q}_{h}(t+\frac{h}{2},\tau))d\tau+
\displaystyle\int_{0}^{1}f(\tilde{q}_{h}(t-\frac{h}{2},\tau))d\tau\big].
\end{aligned}
\end{equation*}

According to the definitions given in \eqref{LLL}, this result can
be rewritten as
\begin{equation*}
\begin{aligned} &L_1(hD)q_{h}(t)=h^2\phi_2(V) L_2(hD)
\displaystyle\int_{0}^{1}f(\tilde{q}_{h}(t,\tau))d\tau,
\end{aligned}
\end{equation*}
which means
\begin{equation*}
\begin{aligned} &L(hD)q_{h}(t)=h^2\phi_2(V)
\displaystyle\int_{0}^{1}f(\tilde{q}_{h}(t,\tau))d\tau.
\end{aligned}
\end{equation*}
By expanding the nonlinear function $f$  at $\xi_h^0(t)$ into its
Taylor series, and comparing the coefficients of
$\mathrm{e}^{\mathrm{i}k\omega t}$, one arrives at
\begin{equation*}
\begin{aligned}  &L(hD)\zeta_h^0(t)=h^2\phi_2(V)\int_{0}^{1}\Big(f(\xi_h^0(t,\tau))+
\sum\limits_{s(\alpha)=0}\frac{1}{m!}f^{(m)}(\xi_h^0(t,\tau))(\xi_h(t,\tau))^{\alpha}\Big)d\tau,\\
&L(hD+\mathrm{i}kh\omega)\zeta_h^k(t)=h^2\phi_2(V)\int_{0}^{1}
\sum\limits_{s(\alpha)=k}\frac{1}{m!}f^{(m)}(\xi_h^0(t,\tau))(\xi_h(t,\tau))^{\alpha}d\tau,\qquad k\neq0,\\
\end{aligned} %
\end{equation*}
where the sum ranges over $m\geq0$,
$\alpha=(\alpha_1,\ldots,\alpha_m)$ with integer $\alpha_i$
satisfying $0<|\alpha_i|<N$,
$s(\alpha)=\sum\limits_{j=1}^{m}\alpha_j,$ and
$(\xi_h(t,\tau))^{\alpha}$ is an abbreviation for
$(\xi_h^{\alpha_1}(t,\tau),\ldots,\xi_h^{\alpha_m}(t,\tau))$. This
formula as well as \eqref{xip} and \eqref{xim} gives the modulation
system for the coefficients $\zeta_h^k(t)$ of the modulated Fourier
expansion $q_n$. Considering the dominate terms in the relations
motivates the following ansatz:
\begin{equation}\label{ansatz}%
\begin{array}
[c]{rll}%
&\ddot{\zeta}^0_{h,1}(t)=G^0_{\pm10}(\cdot)+\cdots, \qquad
&\zeta^0_{h,2}(t)=\frac{1}{
\omega^2}\big(G^0_{\pm20}(\cdot)+\cdots\big),   \\
&\zeta^1_{h,1}(t)=
\frac{-h^2\cos(\frac{h\omega}{2})}{4\sin^2(\frac{h\omega}{2})}\big(G^1_{\pm10}(\cdot)+\cdots\big),
\qquad  & \dot{\zeta}^1_{h,2}(t)=\frac{-\textmd{i}}{2\omega}\sinc(\frac{h\omega}{2})\big(G^1_{\pm20}(\cdot)+\cdots\big),\\
&\zeta^k_{h,1}(t)=
\frac{-h^2\cos(\frac{kh\omega}{2})}{4\sin^2(\frac{k
h\omega}{2})}\big(G^k_{\pm10}(\cdot)+\cdots\big),\qquad
&\zeta^k_{h,2}(t)=\frac{ h^2\phi_2(h
\omega)\cos(\frac{kh\omega}{2})}{-2\sin(\frac{k+1}{2}h\omega)\sin(\frac{k-1}{2}h\omega)
}\big(G^k_{\pm20}(\cdot)+\cdots\big),
\end{array}
\end{equation}
where  the dots  stand  for power series in $\sqrt{h}$. Following
\cite{Hairer00,Hairer16},  we
 truncate the ansatz after the $\mathcal{O}(h^{N+1})$ terms.

Using the scheme of the AAVF method \eqref{AAVFmethod} again, it is
obtained that
\begin{equation*}
\begin{aligned}
 q_{n+1}&=\phi_0(V)q_n+h\phi_1(V)p_n+h\phi_2(V)\phi_1^{-1}(V)(
p_{n+1}+h\Omega^2\phi_1(V)q_n-\phi_0(V)p_n),
\end{aligned}
\end{equation*}
which can be simplified as
$$q_{n+1}-\big(\phi_0(V)+V\phi_2(V)\big)q_n=h\phi_2(V)\phi_1^{-1}(V)p_{n+1}+h\big(\phi_1(V)-\phi_0(V)\phi_1^{-1}(V)\phi_2(V)\big)p_n.$$
According to the definition of $\phi$-functions given by
\eqref{Phi01}, it can be verified straightforwardly that
\begin{equation*}
\begin{aligned}  &\phi_0(V)+V\phi_2(V)=I,\\
&\phi_2(V)\phi_1^{-1}(V)=\phi_1(V)-\phi_0(V)\phi_1^{-1}(V)\phi_2(V)=\tan(\frac{1}{2}h\Omega)(h\Omega)^{-1}.
\end{aligned} %
\end{equation*}
We then obtain \begin{equation}\label{qp connec} q_{n+1}-q_n
=\Omega^{-1}\tan(\frac{1}{2}h \Omega) ( p_{n+1}+ p_n
).\end{equation} By the definition of $L_3$, this relation can be
expressed as $$ L_3(hD)q_h(t)=\Omega^{-1}\tan(\frac{1}{2}h \Omega)
p_h(t).$$ Therefore, we get   the modulation system for the
coefficients $\eta_h^k(t)$ of the modulated Fourier expansion $p_n$
as
\begin{equation}\label{eta modula sys}
\begin{aligned}  & \eta_h^0(t)= \Omega \tan^{-1}(\frac{1}{2}h  \Omega)
L_3(hD)\zeta_h^0(t),\ \   \eta_h^k(t)=\Omega \tan^{-1}(\frac{1}{2}h
\Omega) L_3(hD+\mathrm{i}kh\omega)\zeta_h^k(t)
\end{aligned} %
\end{equation}
for $k\neq0.$ In the light of the Taylor series of $L_3$, one has
the following relationship between $\eta_h^k$ and   $\zeta_h^k$:
\begin{equation}\label{ansatz-1}%
\begin{array}
[c]{rll}%
& \eta^0_{h,1}(t)=\dot{\zeta}^0_{h,1}(t)+\mathcal{O}(h^2), \qquad
&\eta^0_{h,2}(t)=\frac{\cos(\frac{h\omega}{2})}{\sinc(\frac{h\omega}{2})}\dot{\zeta}^0_{h,2}(t)+\mathcal{O}(h),   \\
&\eta^1_{h,1}(t)=\textmd{i}
\omega\frac{\sinc(\frac{h\omega}{2})}{\cos(\frac{h\omega}{2})}\zeta^1_{h,1}(t)+\mathcal{O}(h),
\qquad  & \eta^1_{h,2}=\textmd{i} \omega\zeta^1_{h,2}+\mathcal{O}(h),\\
&\eta^k_{h,1}(t)=\textmd{i}k
\omega\frac{\sinc(\frac{kh\omega}{2})}{\cos(\frac{kh\omega}{2})}
\zeta ^k_{h,1}(t)+\mathcal{O}(h),\qquad &\eta^k_{h,2}(t)=\textmd{i}
\omega\frac{\tan(\frac{kh\omega}{2})}{\tan(\frac{h\omega}{2})}\zeta^k_{h,2}+\mathcal{O}(h),
\end{array}
\end{equation}
where $\abs{k}>1.$ This presents the modulation equation of
$\eta_h^k$.
 \vskip1mm

 \textbf{II. Initial values.} By the conditions that
\eqref{MFE-ERKN} is satisfied without  the remainder term for $t =
0$ and $t = h$,  the initial values for the differential equations
of $\zeta^0_{h,1}$ and $\zeta^1_{h,2}$ can be determined as follows.

Considering the conditions  $p_{h}(0)=p_0$ and
 $q_{h}(0)=q_0$,
we get
\begin{equation*}
\begin{aligned}
&p_{0,1}=\eta_{h,1}^0(0)+\mathcal{O}(h)=\dot{\zeta}^0_{h,1}(0)+\mathcal{O}(h),\\
&q_{0,1}=\zeta_{h,1}^0(0)+\mathcal{O}(h),\
 \ q_{0,2}=2\mathrm{Re}(\zeta^1_{h,2}(0))+\mathcal{O}(h^2).\end{aligned} %
\end{equation*}
This gives the initial values
$\zeta_{h,1}^0(0),\dot{\zeta}^0_{h,1}(0)$  and
$\mathrm{Re}(\zeta^1_{h,2}(0))$. Moreover, it follows from
\eqref{energy bound} that $q_{0,2}=\mathcal{O}(\omega^{-1}),$ which
implies that $\mathrm{Re}(\zeta^1_{h,2}(0))=\mathcal{O}(h)$. In what
follows, we derive the value of $\mathrm{Im}(\zeta^1_{h,2}(0))$.

From $ q_{h,1}(h)=q_{1,1},\ q_{h,2}(h)=q_{1,2}$ and the first
formula of AAVF method, it follows that
\begin{equation*}
q_{1,2}-\cos(h\omega)q_{0,2}=h \mathrm{sinc}(h\omega)p_{0,2}+
\mathcal{O}(h^2).\end{equation*}  We compute
\begin{equation*}
\begin{aligned}
&q_{1,2}-\cos(h\omega)q_{0,2}=q_{h,2}(h)-\cos(h\omega)q_{h,2}(0)\\
=&\sum\limits_{|k|<N} \mathrm{e}^{\mathrm{i}k\omega
h}\zeta_{h,2}^k(h)-\cos(h\omega)\sum\limits_{|k|<N}
\zeta_{h,2}^k(0)\\
=&\zeta_{h,2}^0(h)+ \mathrm{e}^{\mathrm{i}\omega h}\zeta_{h,2}^1(h)+
\mathrm{e}^{-\mathrm{i}\omega
h}\zeta_{h,2}^{-1}(h)-\cos(h\omega)\Big(\zeta_{h,2}^0(0)+
\zeta_{h,2}^1(0)+\zeta_{h,2}^{-1}(0)\Big)+\mathcal{O}(h^2).
\end{aligned}
\end{equation*}
Expanding the functions $\zeta_{h,2}^0(h),\ \zeta_{h,2}^{1}(h),\
\zeta_{h,2}^{-1}(h)$ at $h=0$ yields
\begin{equation*}
\begin{aligned}
q_{1,2}-\cos(h\omega)q_{0,2}=&(1-\cos(h\omega))\zeta_{h,2}^0(0)+\mathrm{i}\sin(h\omega)
(\zeta_{h,2}^1(0)-\zeta_{h,2}^{-1}(0)) +\mathcal{O}(h^2).
\end{aligned}
\end{equation*}
 It is clear from
$1-\cos(h\omega)=\frac{1}{2}h^2\omega^2\mathrm{sinc}^2(h\omega/2)$
that
\begin{equation*}
\begin{aligned}
&(1-\cos(h\omega))\zeta_{h,2}^0(0)=\frac{1}{2}h^2\omega^2\mathrm{sinc}^2(h\omega/2)\zeta_{h,2}^0(0)=2\sin^2(h\omega/2)\zeta_{h,2}^0(0)=\mathcal{O}(h^2).
\end{aligned}
\end{equation*}
 Thus it is confirmed
that
\begin{equation*}
\begin{aligned}
\mathrm{i}\sin(h\omega) (\zeta_{h,2}^1(0)-\zeta_{h,2}^{-1}(0))  =&h
\mathrm{sinc}(h\omega)p_{0,2}+\mathcal{O}(h^2),
\end{aligned}
\end{equation*}
which yields
$2\mathrm{Im}(\zeta_{h,2}^1(0))=\omega^{-1}p_{0,2}+\mathcal{O}(h)=\mathcal{O}(h).$

 \vskip1mm
\textbf{III. Bounds of the coefficients functions.} Based on
Assumption \ref{ass}, the ansatz given by \eqref{ansatz} and
\eqref{ansatz-1}, and the initial values presented in the above
part, the bounds shown in  \eqref{coefficient func} are easily
derived.

 \vskip1mm
 \textbf{IV. Remainder.}  For $t=nh$, let
\begin{equation*}
\begin{aligned}
 \delta_q(t+h)=& q_{h}(t+h)-\phi_0(V)q_{h}(t)-h\phi_1(V)p_{h}(t)\\
 &-h^2\phi_2(V)\displaystyle\int_{0}^{1}f((1-\tau)q_{h}(t)+\tau q_{h}(t+h))d\tau,\\
  \delta_p(t+h)= &p_{h}(t+h)+h\Omega^2\phi_1(V)q_{h}(t)-\phi_0(V)p_{h}(t)\\
  &-h\phi_1(V)\displaystyle\int_{0}^{1}f((1-\tau)q_{h}(t)+\tau
q_{h}(t+h))d\tau.
\end{aligned}
\end{equation*}
It is clear from the two-step formulation that
$\delta_q(t+h)+\delta_q(t-h)=\mathcal{O}(h^{N+2})$.  According to
the choice for the initial values, we obtain
$\delta_q(0)=\mathcal{O}(h^{N+2}).$ Thus it is derived that
$\delta_q(t)=\mathcal{O}(h^{N+2})+\mathcal{O}(th^{N+1}).$  Then
according to \eqref{qp connec}, one gets
$\delta_p=\mathcal{O}(h^{N}).$ By letting $R_n=q_n-q_{h}(t)$ and
$S_n=p_n-p_{h}(t),$ we obtain the following  error recursion
\begin{equation*}
\begin{aligned}
&\left(
  \begin{array}{c}
     \Omega R_{n+1} \\
     S_{n+1}\\
  \end{array}
\right) =\left(
                                                           \begin{array}{cc}
                                                             \cos(h\Omega) & \sin(h\Omega) \\
                                                             -\sin(h\Omega) & \cos(h\Omega) \\
                                                           \end{array}
                                                         \right)\left(
  \begin{array}{c}
     \Omega R^{n} \\
     S^{n}\\
  \end{array}
\right)\\
+&h\left(
          \begin{array}{c}
            h \Omega \phi_2(V)\displaystyle\int_{0}^{1}\big(f((1-\tau)q_{n}+\tau q_{n+1})-f((1-\tau)q_{h}(t)+\tau q_{h}(t+h))\big)d\tau \\
            \phi_1(V)\displaystyle\int_{0}^{1}\big(f((1-\tau)q_{n}+\tau q_{n+1})-f((1-\tau)q_{h}(t)+\tau q_{h}(t+h))\big)d\tau \\
          \end{array}
        \right)\\
        +&\left(
          \begin{array}{c}
          \Omega \delta_q\\
    \delta_p \\
          \end{array}
        \right).
\end{aligned}
\end{equation*}
By using the Lipschitz continuous of the nonlinearity,   one obtains
$$\norm{f((1-\tau)q_{n}+\tau q_{n+1})-f((1-\tau)q_{h}(t)+\tau q_{h}(t+h))}\leq\norm{R_n}+\norm{R_{n+1}}.$$
Then the remainder \eqref{remainder} can be derived by solving the
error recursion and the application of a discrete Gronwall
inequality.

 The proof of this theorem is complete.
 \hfill\end{proof}

\section{Long-time oscillatory energy conservation} \label{sec:oscillatory energy conservation}
This section is devoted to  showing the long-time oscillatory energy
conservation of the AAVF method.

Denote $\zeta=\big(\zeta^{-N+1}_h,\cdots,\zeta^{-1}_h,
\zeta^{0}_h,\zeta^{1}_h,\cdots,\zeta^{N-1}_h\big). $ The modulation
functions of the AAVF method  have the following almost invariant.
\begin{theo}\label{second invariant thm}
Suppose that the conditions of Theorem \ref{energy thm} hold. For
the coefficient functions of   the modulated Fourier expansion,
there exists a function $\widehat{\mathcal{I}}[\zeta]$ such that
\begin{equation*}
\widehat{\mathcal{I}}[\zeta](t)=\widehat{\mathcal{I}}[\zeta](0)+\mathcal{O}(th^{N}),
\end{equation*}
where $0\leq t\leq T.$ Moreover, this almost invariant can be
expressed as
\begin{equation*}
\widehat{\mathcal{I}}[\zeta] =2\omega^2\frac{
\cos(\frac{1}{2}h\omega) }{\sinc(\frac{1}{2}h\omega)}
\big(\zeta_{h,2}^{-1}\big)^\intercal\zeta_{h,2}^{1}
+\mathcal{O}(h^2).
\end{equation*}
\end{theo}
\begin{proof}
 With the proof of Theorem \ref{energy
thm} proposed in the previous section, one obtains
\begin{equation*}
\begin{aligned}
L(hD)q_{h}(t)=h^2\phi_2(V)
\displaystyle\int_{0}^{1}f(\tilde{q}_{h}(t,\tau))d\tau+\mathcal{O}(h^{N+2}),
\end{aligned}
\end{equation*}
where  we use the following denotations:
\begin{equation*}
\begin{aligned}q_{h}(t)=\sum\limits_{ |k|<N}q^k_{h}(t),\ \ \tilde{q}_{h}(t,\tau)=\sum\limits_{
|k|<N}\tilde{q}^k_{h}(t,\tau).
\end{aligned}
\end{equation*}
Here $q^k_{h}$ and $\tilde{q}^k_{h}$ are defined as
$q^k_{h}(t)=\mathrm{e}^{\mathrm{i}k\omega t}\zeta_h^k(t)$ and $
\tilde{q}^k_{h}(t,\tau)=\mathrm{e}^{\mathrm{i}k\omega
t}\xi_h^k(t,\tau),$ respectively.  By considering the  definitions
of $\tilde{q}_{h},\ q_h$ and comparing the coefficients of
$\mathrm{e}^{\mathrm{i}k\omega t}$, we obtain   the  equations in
terms of $q_h^k:$
\begin{equation*}
L(hD)q^k_{h}(t)=-h^2\phi_2(V) \displaystyle\nabla_{-k}
\mathcal{U}(\tilde{q}(t,\tau))+\mathcal{O}(h^{N+2}),
\end{equation*}
where $\mathcal{U}(\tilde{q}(t,\tau))$ is defined as
\begin{equation}
\begin{aligned}
&\mathcal{U}(\tilde{q}(t,\tau))=\int_{0}^{1}U(\tilde{q}^0_h(t,\tau))d\tau+
\sum\limits_{s(\alpha)=0}\frac{1}{m!}\int_{0}^{1}U^{(m)}(\tilde{q}^0_h(t,\tau))
(\tilde{q}_h(t,\tau))^{\alpha}d\tau,
\end{aligned}
\label{newuu}%
\end{equation}
and $\tilde{q}(t,\tau)$ is given by
\begin{equation*}
\begin{aligned}
\tilde{q}(t,\tau)=\big(\tilde{q}^{-N+1}_h(t,\tau),\ldots,
\tilde{q}^{0}_h(t,\tau),\ldots,\tilde{q}^{N-1}_h(t,\tau)\big).
\end{aligned}
\end{equation*}

  Define  a vector function
$\tilde{q}(\lambda,t,\tau)$ of $\lambda$ as below
$$\tilde{q}(\lambda,t,\tau)=\big( \mathrm{e}^{\mathrm{i}(-N+1)\lambda
\omega}\tilde{q}^{-N+1}_h(t,\tau),\cdots,
\tilde{q}^{0}_h(t,\tau),\cdots,\mathrm{e}^{\mathrm{i}(N-1)\lambda
\omega}\tilde{q}^{N-1}_h(t,\tau)\big).$$ It can be observed from the
definition \eqref{newuu}
 that $\mathcal{U}( \tilde{q}(\lambda,t,\tau))$ is  independent of
$\lambda$ and $\tau$. Thus, considering its derivative with respect
to $\lambda$ implies
\begin{equation*}
\begin{aligned}0=& \frac{\partial}{\partial\lambda}  \mathcal{U}(
\tilde{q}(\lambda,t,\tau))  =\Big(\frac{\partial}{\partial\tilde{q}}
 \mathcal{U}( \tilde{q}(\lambda,t,\tau)) \Big)^\intercal
\frac{\partial}{\partial\lambda }\tilde{q}(\lambda,t,\tau)\\ =&
\sum\limits_{|k|<N}\mathrm{i}k\omega\mathrm{e}^{\mathrm{i}k\lambda
\omega} (\tilde{q}^{k}_h(\lambda,t,\tau))^\intercal\nabla_{k}
  \mathcal{U}( \tilde{q}(\lambda,t,\tau)).\end{aligned}
\end{equation*}
The choice of $\lambda=0$ and $\tau=\frac{1}{2}$ yields $
\sum\limits_{|k|<N}\mathrm{i}k\omega
(\tilde{q}^{k}_h(t,\frac{1}{2}))^\intercal \nabla_{ k} \mathcal{U}(
\tilde{q}(t,\frac{1}{2}))=0. $ Therefore, one gets
\begin{equation}
\begin{aligned}
0=&\sum\limits_{|k|<N}\mathrm{i}k\omega
(\tilde{q}^{-k}_h(t,\frac{1}{2}))^\intercal \nabla_{ -k}
\mathcal{U}( \tilde{q}(t,\frac{1}{2}))\\
=& \sum\limits_{|k|<N}\mathrm{i}k\omega
(\tilde{q}^{-k}_h(t,\frac{1}{2}))^\intercal\frac{1}{-h^2}
\phi^{-1}_2(V) L(hD)q^k_{h}(t)+\mathcal{O}(h^{N}).
\end{aligned}\label{duu-00}
\end{equation}
Inserting the expressions of $q^k_{h}$ and $\tilde{q}^k_{h}$ into
\eqref{duu-00} gives
\begin{equation}
\begin{aligned}
\mathcal{O}(h^{N})&=\sum\limits_{|k|<N}\mathrm{i}k\omega
(\xi_h^{-k}(t,\frac{1}{2}))^\intercal\frac{1}{-h^2} \phi^{-1}_2(V)
L(hD+\mathrm{i}k\omega h)\zeta^k_{h}(t) \\
&=\sum\limits_{|k|<N}\mathrm{i}k\omega (L_4(hD,\frac{1}{2},-k)
\zeta_h^{-k}(t))^\intercal\frac{1}{-h^2} \phi^{-1}_2(V)
L(hD+\mathrm{i}k\omega h)\zeta^k_{h}(t).
\end{aligned}
\label{duu-I}%
\end{equation}
By Proposition \ref{lhd pro}, we get
\begin{equation*}
\begin{aligned}
&L_4(hD, \frac{1}{2},-k)
\zeta_h^{-k}(t)=(\cdot)\bar{\zeta}^k_{h}+\textmd{i}h(\cdot)\dot{\bar{\zeta}}^k_{h}+h^2(\cdot)
\ddot{\bar{\zeta}}^k_{h}+\cdots, \\
&L(hD+\mathrm{i}k\omega
h)\zeta^k_{h}=(\cdot)\zeta^k_{h}+\textmd{i}h(\cdot)\dot{\zeta}^k_{h}+h^2(\cdot)
\ddot{\zeta}^k_{h}+\cdots.
\end{aligned}
\end{equation*}
Looking closer to the right-hand side of  \eqref{duu-I}, using the
above expressions of  $L_4$ and $L$,  and considering the  formulae
on p. 508 of \cite{hairer2006},
 it can be verified that the right-hand side of  \eqref{duu-I}  is a total
derivative.
 Therefore,  there exists a function $\widehat{\mathcal{I}}$ such
 that
$\frac{d}{dt}\widehat{\mathcal{I}}[\zeta](t)=\mathcal{O}(h^{N})$. An
integration of it immediately   implies the first  statement of the
theorem.

By the previous analysis and  the bounds of Theorem \ref{energy
thm},  the construction of $\widehat{\mathcal{I}}$  is obtained as
follows:
\begin{equation*}
\begin{aligned}\widehat{\mathcal{I}}[\zeta]=&2\frac{2h\omega  \sin(\frac{1}{2}h\omega)\cos(\frac{1}{2}h\omega)
}{h^2\phi_2(h\omega)}\frac{1}{2}\big(\zeta_{h,2}^{-1}\big)^\intercal\zeta_{h,2}^{1}
+\mathcal{O}(h^2) \\
=&2\omega^2\frac{  \cos(\frac{1}{2}h\omega)
}{\sinc(\frac{1}{2}h\omega)}
\big(\zeta_{h,2}^{-1}\big)^\intercal\zeta_{h,2}^{1}
+\mathcal{O}(h^2).
\end{aligned}
\end{equation*}
We complete the proof of this theorem.
 \hfill
\end{proof}

We are now in a position to present the main result of this paper.

\begin{theo}\label{HHthm}Define $\mathcal{I} [\zeta]
=\widehat{\mathcal{I}}[\zeta]/\sigma(h\omega), $
 where  $\sigma(h\omega)$ is given by $\sigma(h\omega)=\frac{  \cos(\frac{1}{2}h\omega)
}{\sinc(\frac{1}{2}h\omega)}.$ Under the conditions of Theorem
\ref{energy thm} and that $\abs{\cos(\frac{1}{2}h\omega)}\geq c h^m$
for  some $m$, we have the following relation between
$\mathcal{I}[\zeta]$ and $I(q_n,p_n)$:
\begin{equation*}
\begin{aligned}
&\mathcal{I}[\zeta](nh)=I(q_n,p_n)+\mathcal{O}(h).
\end{aligned}
\end{equation*}
Moreover, it holds that
\begin{equation*}
I(q_n,p_n)=I(q_0,p_0)+\mathcal{O}(h)
\end{equation*}
for $0\leq nh\leq h^{-N+1}.$ The constants symbolized by
$\mathcal{O}$ are independent of $n,h, \omega$,  but depend on $N,
T$ and the constants in the assumptions.

\end{theo}
\begin{proof}
According to the definition of $\mathcal{I}$ and under the
conditions of this theorem,   one obtains
 \begin{equation}\label{HIW}\mathcal{I} [\zeta]
=2\omega^2  \big(\zeta_{h,2}^{-1}\big)^\intercal\zeta_{h,2}^{1}
+\mathcal{O}(h^2). \end{equation}
 On the other hand, it follows from
\eqref{coefficient func} that $
\eta_{h,2}^{\pm1}(t)=\pm\mathrm{i}\omega
 \zeta_{h,2}^{\pm1}(t) +\mathcal{O}(h).
$ Thus using the bounds of Theorem \ref{energy thm}, we have
\begin{equation*}
\begin{aligned}  \omega q_{n,2}& =\omega \big(\mathrm{e}^{\mathrm{i}\omega t}\zeta_{h,2}^1(t)+\mathrm{e}^{-\mathrm{i}\omega
t}\zeta_{h,2}^{-1}(t)\big)+\mathcal{O}(h),\\
 p_{n,2} &=\mathrm{i}\omega \big(\mathrm{e}^{\mathrm{i}\omega t}\zeta_{h,2}^1(t)-\mathrm{e}^{-\mathrm{i}\omega
t}\zeta_{h,2}^{-1}(t)\big)+\mathcal{O}(h).
\end{aligned} %
\end{equation*}
This implies
 \begin{equation}\label{HIE} \begin{aligned}I(q_n, p_n)&=\frac{1}{2}p_{n,2}^\intercal p_{n,2} +\frac{1}{2}\omega^2q_{n,2}^\intercal q_{n,2}\\
 &=\frac{1}{2}\norm{\mathrm{i}\omega \big(\mathrm{e}^{\mathrm{i}\omega t}\zeta_{h,2}^1(t)-\mathrm{e}^{-\mathrm{i}\omega
t}\zeta_{h,2}^{-1}(t)\big)}^2 +\frac{1}{2}\norm{\omega
\big(\mathrm{e}^{\mathrm{i}\omega
t}\zeta_{h,2}^1(t)+\mathrm{e}^{-\mathrm{i}\omega
t}\zeta_{h,2}^{-1}(t)\big)}^2 \\
 &=2\omega^2
\big(\zeta_{h,2}^{-1}\big)^\intercal\zeta_{h,2}^{1} +\mathcal{O}(h),
 \end{aligned}\end{equation}
where we have used  the fact that
$\norm{v+\bar{v}}^2+\norm{v-\bar{v}}^2=4\norm{v}^2$. A comparison
between \eqref{HIW} and \eqref{HIE}  gives the first stated relation
of this theorem. Following  the identical argument  given in Section
XIII of \cite{hairer2006},   the result of the long-time oscillatory
energy preservation  can be obtained by patching together many
intervals of length $h$.
\end{proof}

\begin{rem}From the analysis stated above for  oscillatory energy conservation, it follows that the result
of Theorem  \ref{HHthm} cannot be improved even high
 order quadratures are chosen for the AAVF method
\eqref{AAVFmethod}, which explains the numerical phenomenon shown in
Section \ref{sec:Formulation}.
\end{rem}

\section{Generalization of multi-frequency case} \label{sec: multi-frequency case}
In this section,  we are devoted  to extending the analysis to a
muti-frequency highly oscillatory Hamiltonian system  with the
following Hamiltonian function
 \begin{equation}\label{H mul}
\begin{array}
[c]{ll}%
H(q,p) &
=\dfrac{1}{2}\textstyle\sum\limits_{j=0}^{l}\Big(\norm{p_j}^{2}+\dfrac
{\lambda_j^{2}}{\epsilon^2}\norm{q_j}^{2}\Big)+U(q),
\end{array}
\end{equation}
where $q=(q_{0},q_1,\ldots,q_l),\ p=(p_{0},p_1,\ldots,p_l)$ with
$q_j,\ p_j\in \mathbb{R}^{d_j}$, $\lambda_0=0$ and $\lambda_j\geq1$
are distinct real numbers for $j\geq1$, $\epsilon$ is a small
positive parameter, and $U(q)$ is a smooth potential function. It is
well known that this system has the oscillatory energy of the $j$th
frequency as
\begin{equation*}I_j(q,p)=\dfrac{1}{2} \Big(\norm{p_j}^{2}+\dfrac
{\lambda_j^{2}}{\epsilon^2}\norm{q_j}^{2}\Big),
\end{equation*}
 and its total oscillatory energy is
$
I (q,p)=\sum\limits_{j=1}^{l}I_j(q,p). 
$

Muti-frequency  highly oscillatory Hamiltonian system often arises
  in a wide range of applications, such as in physics
and engineering,   astronomy, molecular dynamics, and in problems of
wave propagation in classical and quantum physics. There have been
many efficient numerical methods for solving this system and we
refer to
\cite{hairer2006,Hochbruck2010,Hochbruck2009,wang-2016,wang2017-ANM,wang2017-Cal,wubook2018,wu2013-book}
as well as the references contained therein.
 This muti-frequency
Hamiltonian system  can also be rewritten as the highly oscillatory
second-order system \eqref{prob} with
 $\Omega=\textmd{diag}(\omega_0
I_{d_0},\omega_1I_{d_1},\ldots, \omega_lI_{d_l})$, where
$\omega_j=\lambda_j/\epsilon$. Thus the AAVF method
\eqref{AAVFmethod} can be used to solve this system. In what
follows, we briefly discuss the long-time oscillatory energies
conservations of the AAVF method for this muti-frequency highly
oscillatory Hamiltonian system. The technique used here is the
muti-frequency modulated Fourier expansion of the AAVF method, which
can be obtained by the generalization of Sections \ref{sec:mfe of
the methods}-\ref{sec:oscillatory energy conservation} of this paper
and following the way used in \cite{Cohen05}. For brevity, we just
present the main results and omit the details of proof.

\subsection{The main results for multi-frequency case} Let
\begin{equation*}
\begin{aligned}&\lambda=(\lambda_1,\ldots,\lambda_l),\quad
k=(k_1,\ldots,k_l),\quad k\cdot
\lambda=k_1\lambda_1+\cdots+k_l\lambda_l,
\end{aligned}
\end{equation*}
and denote the resonance module by
\begin{equation}\mathcal{M}= \{k\in \mathbb{Z}^{l}:\ k\cdot
\lambda=0\}.
\label{M mul}%
\end{equation}
Following \cite{Cohen05}, we use the following notations
\begin{equation*}
\begin{aligned}
&\omega=(\omega_1,\ldots,\omega_l),\quad \langle
j\rangle=(0,\ldots,1,\ldots,0),\quad |k|=|k_1|+\cdots+|k_l|.
\end{aligned}
\end{equation*}
For  the resonance   module \eqref{M mul}, denote by $\mathcal{K}$
the set of representatives of the equivalence classes in
$\mathbb{Z}^l\backslash \mathcal{M}$ which are chosen such that for
each $k\in\mathcal{K}$ the sum $|k|$ is minimal in the equivalence
class $[k] = k +\mathcal{M},$ and that with $k\in\mathcal{K}$, also
$-k\in\mathcal{K}.$   For the positive integer $N$, we let
\begin{equation*}\mathcal{N}=\{k\in\mathcal{K}:\ |k|\leq N\},\ \ \ \ \ \mathcal{N}^*=\mathcal{N}\backslash
\{(0,\ldots,0)\}.
\end{equation*}

The multi-frequency modulated Fourier expansion  of the AAVF method
is presented in the following theorem.
\begin{theo}\label{energy thm mul}
The initial values  are supposed to  satisfy $H(q_0,p_0)\leq E$.
Assume that $ h/\epsilon \geq c_0
> 0 $
 and  the following numerical non-resonance condition  is true
\begin{equation*}
|\sin(\frac{h}{2 \epsilon}(k\cdot \lambda))| \geq c \sqrt{h}\ \
\mathrm{for} \ \ k \in \mathbb{Z}^l\backslash \mathcal{M}
   \ \   \mathrm{with} \ \  |k|\leq N
\end{equation*}
for some $N\geq2$ and $c>0$. Then   the AAVF method  admits the
following multi-frequency  modulated Fourier expansion
\begin{equation*}
\begin{aligned} &q_{n}= \zeta(t)+\sum\limits_{k\in\mathcal{N}^*} \mathrm{e}^{\mathrm{i}(k \cdot \omega) t}\zeta^k(t)+\mathcal{O}(h^{N}),\ \
p_{n}= \eta(t)+\sum\limits_{k\in\mathcal{N}^*} \mathrm{e}^{\mathrm{i}(k \cdot \omega) t}\eta^k(t)+\mathcal{O}(h^{N-1}),\\
\end{aligned}
\end{equation*}
for $0 \leq t=nh \leq T$. The coefficient functions as well as all
their derivatives are bounded by
\begin{equation*}
\begin{array}
[c]{rll}%
&\zeta_0(t)=\mathcal{O} (1  ),\qquad \qquad \ &\eta_0(t)= \mathcal{O}(1 ),  \\
&\zeta_j(t)=\mathcal{O}(h^2),\qquad \ \
&\eta_j(t))=\mathcal{O}(h),  \\
  &\zeta_j^{\pm\langle j\rangle}(t)=\mathcal{O}(h),  \
&\eta_j^{\pm\langle j\rangle}(t)=\mathcal{O}(1),\\
&\zeta_0^k(t)=\mathcal{O}\big(h^{|k|+1}\big),
&\eta_0^k(t)=\mathcal{O}\big(h ^{|k|}\big),\ k\in\mathcal{N}^*, \\
&\zeta_j^k(t)=\mathcal{O}\big(h^{|k|+1}\big),\
 &\eta_j^k(t)=\mathcal{O}\big(h^{|k|}\big),\ \ k\neq\pm\langle j\rangle,
\end{array}
\end{equation*}
for $j=1,\ldots,l$.
\end{theo}

An almost-invariant is obtained for the functions of the
multi-frequency modulated Fourier expansion.
\begin{theo}\label{second invariant thm mul}
Under the conditions of Theorem \ref{energy thm mul}, there exists a
function $\widehat{\mathcal{I}}[\zeta]$ such that
\begin{equation*}
\widehat{\mathcal{I}}_{\mu}[\zeta](t)=\widehat{\mathcal{I}}_{\mu}[\zeta](0)+\mathcal{O}(th^{N})+\mathcal{O}(t\epsilon^{M-1})
\end{equation*}
for all $\mu\in \mathbb{R}^l$ and $0\leq t\leq T.$ Here $M=\min\{|
k|:0\neq k\in \mathcal{M}\}$. The almost-invariant satisfies
\begin{equation*}
\widehat{\mathcal{I}}_{\mu}[\zeta](t)=\widehat{\mathcal{I}}_{\mu}[\zeta](0)+\mathcal{O}(th^{N})
\end{equation*}
 for $ \mu\perp \mathcal{M}_N:=\{k\in\mathcal{M}:\
|k|\leq N\}$ and $0\leq t\leq T.$
 Moreover, $\widehat{\mathcal{I}}_{\mu}$ can be expressed in
\begin{equation*}\begin{aligned}
\widehat{\mathcal{I}}_{\mu}[\zeta] =&\sum\limits_{j=1}^l2\omega_j^2
\frac{\mu_j}{\lambda_j}\frac{\cos(\frac{1}{2}h\omega_j)
}{\sinc(\frac{1}{2}h\omega_j)}\big(\zeta_j^{-\langle
j\rangle}\big)^\intercal\zeta_j^{\langle j\rangle} +\mathcal{O}(h).
\end{aligned}
\end{equation*}
\end{theo}

Consider  the following  modified oscillatory energies
\begin{equation*}
\begin{aligned}
& I_{\mu}^{*}(q,p)=  \sum\limits_{j=1}^l
\sigma(\xi_j)\frac{\mu_j}{\lambda_j} I_j(q,p),
\end{aligned}
\end{equation*}
 where $\sigma$ is defined as
$
 \sigma(\xi_j):=\frac{\cos(\frac{1}{2}\xi_j)
}{\sinc(\frac{1}{2}\xi_j)}.$
 We then  obtain the result about the long-time modified oscillatory energies conservations of the AAVF method for multi-frequency
 highly oscillatory systems.
\begin{theo}\label{HHthm-new nul} Under the
conditions of Theorem  \ref{energy thm mul}, we have
\begin{equation*}
\begin{aligned}
&\widehat{\mathcal{I}}_{\mu}[\vec{\zeta},\vec{\eta}](nh)=
I_{\mu}^{*}(q_n,p_n)+\mathcal{O}(h).
\end{aligned}
\end{equation*}
Moreover,   it   holds that
  \begin{equation*}
\begin{aligned}
I_{\mu}^{*}(q^n,p^n)&=I_{\mu}^{*}(q^0,p^0)+\mathcal{O}(h)\\
\end{aligned}
\end{equation*}
for $0\leq nh\leq h^{-N+1}$,   $\mu\in \mathbb{R}^l$ and $ \mu\perp
\mathcal{M}_N$. The constants symbolised by $\mathcal{O}$ are
independent of $n, h, \Omega$, but depend on $N, T$ and the
constants in the assumptions.
\end{theo}

\subsection{Numerical experiments}
In order to  illustrate the numerical conservation of
 the modified oscillatory energies for the AAVF method, we consider a Hamiltonian \eqref{H mul} with
$l=3$ and $\lambda=(1,\sqrt{2},2)$  (see \cite{Cohen05}). It is
shown in  \cite{Cohen05} that there is the $1:2$ resonance between
$\lambda_1$ and $\lambda_3$: $\mathcal{M}= \{(-2k_3,0,k_3):\ k_3\in
\mathbb{Z}\}.$ For this problem,  the dimension of $q_1 = (q_{11},
q_{12})$ is assumed to be 2 and all the other $q_j$ are assumed to
be 1.  We consider   $ \epsilon^{-1}= \omega = 70$, the potential
$U(q) = (0.001q_0 + q_{11} + q_{22} + q_2 + q_3)^ 4,$ and
$$q(0) = (1,0.3\epsilon,0.8\epsilon,-1.1\epsilon,0.7\epsilon),\ \ p(0) = (-0.75,0.6,0.7,-0.9,0.8)$$
as initial values. For $\lambda=(1,\sqrt{2},2),$ it is chosen that
$\mu=(1,0,2)$ and $\mu=(0,\sqrt{2},0)$
 for $I_{\mu}$ and  the corresponding results  are
$I_{\mu}=I_1+I_3$ and $I_{\mu}=I_2.$ We   integrate this problem on
the interval $[0,10000]$ with $h=0.1, 0,01$.  The modified
oscillatory energies
  conservations  are shown in Figs.
\ref{p3-0}-\ref{p4}.

 \begin{figure}[ptb]
\centering
\includegraphics[width=6cm,height=3cm]{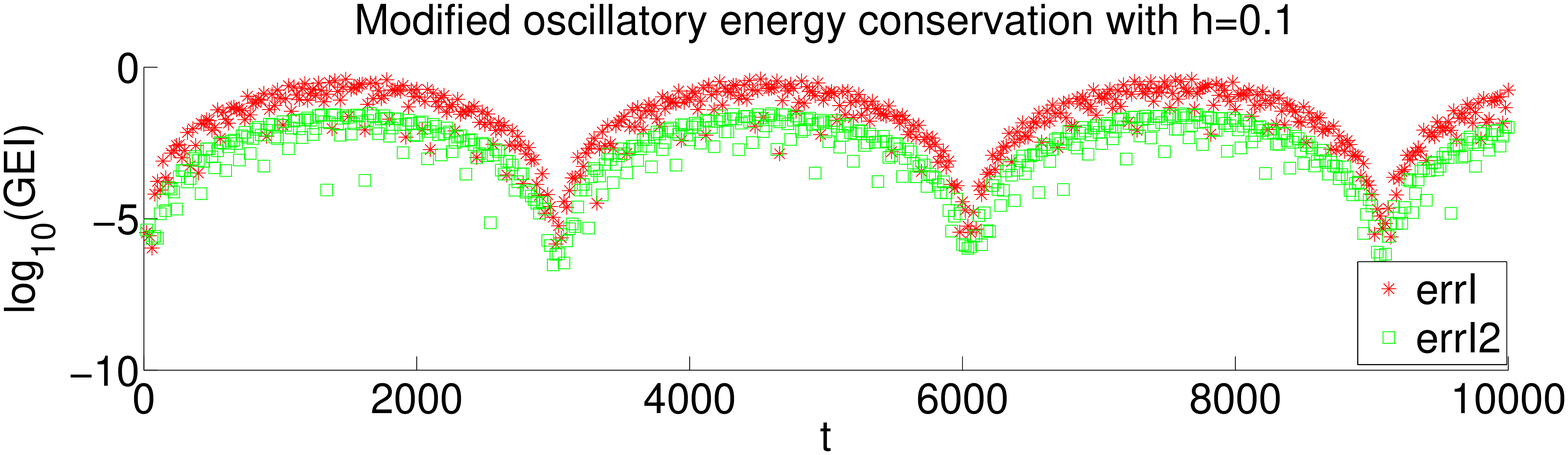}
\includegraphics[width=6cm,height=3cm]{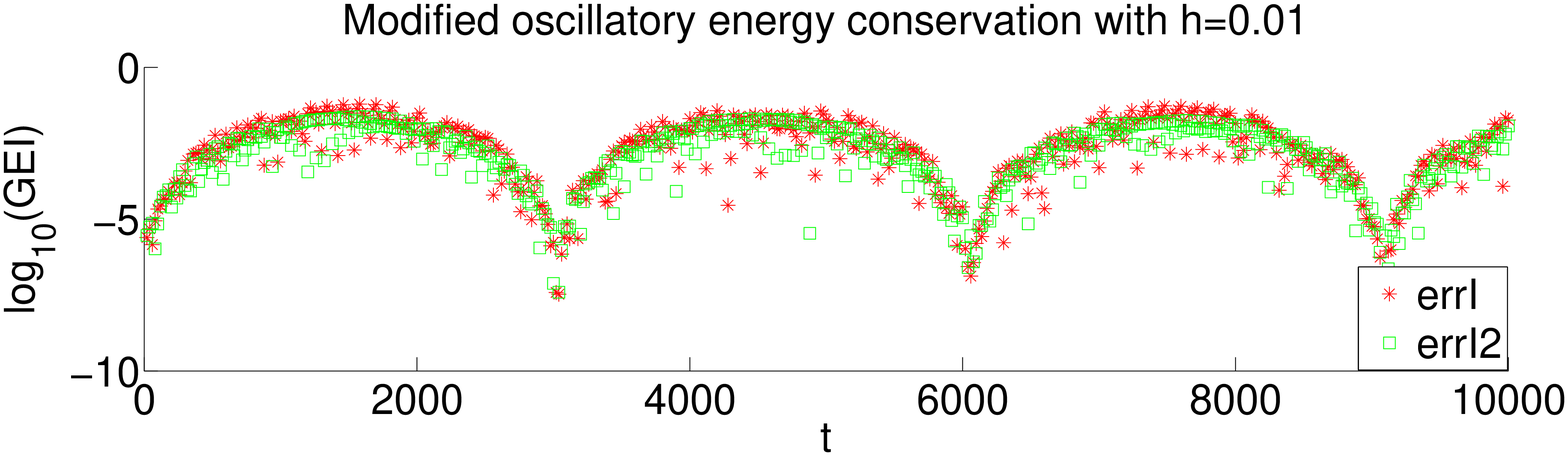}
\caption{AAVF 1: the logarithm of the modified oscillatory energy
errors against $t$.} \label{p3-0}
\end{figure}

 \begin{figure}[ptb]
\centering
\includegraphics[width=6cm,height=3cm]{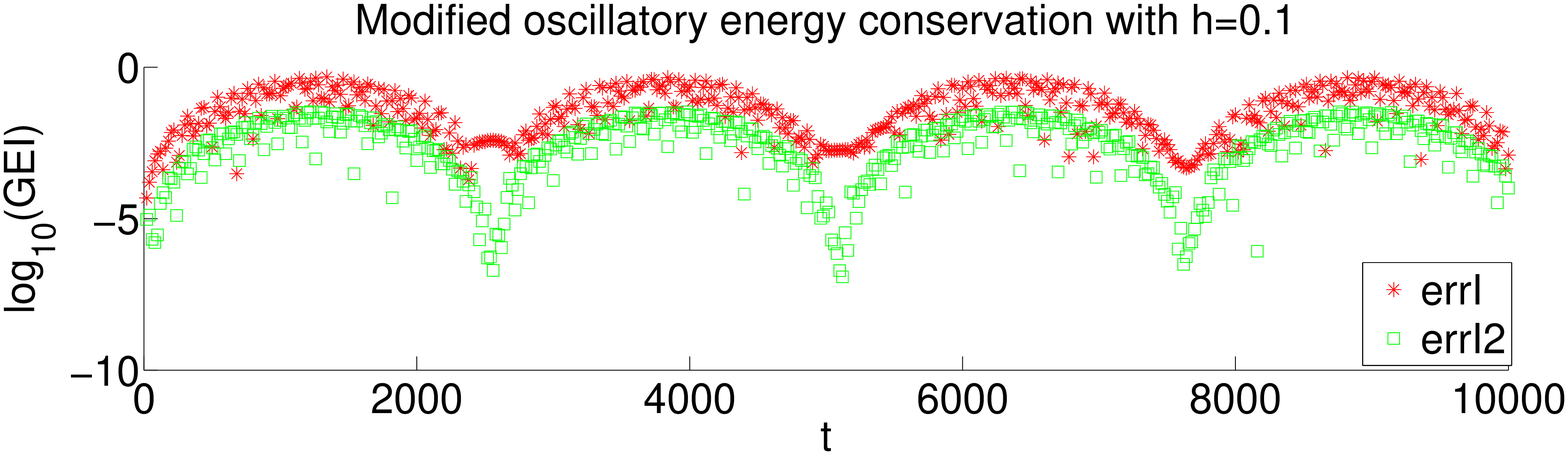}
\includegraphics[width=6cm,height=3cm]{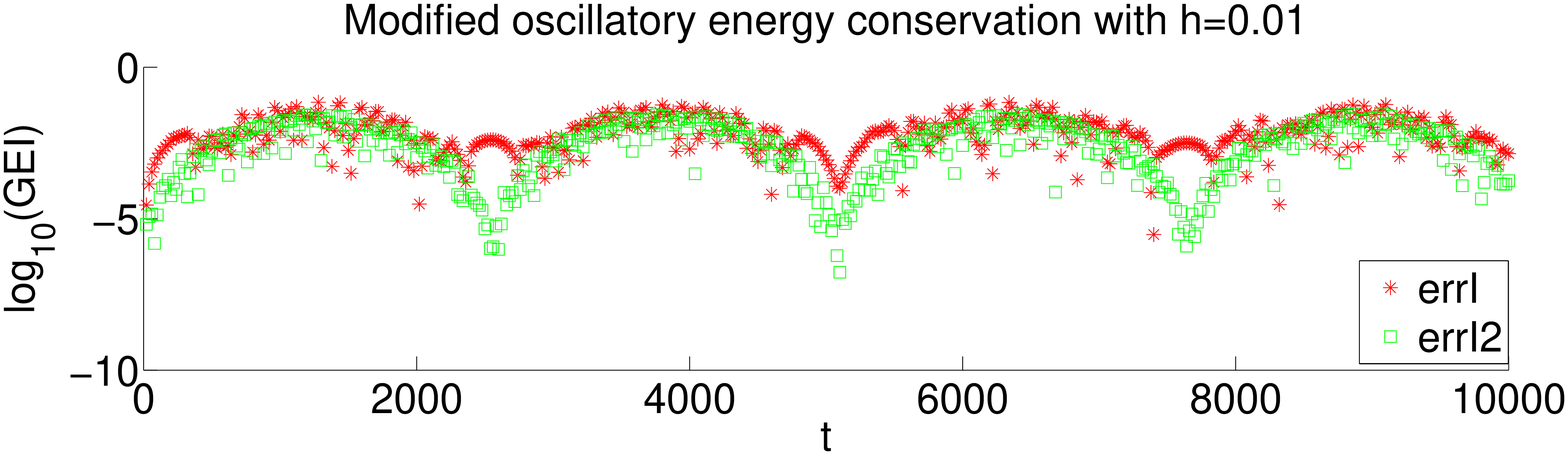}
\caption{AAVF 2: the logarithm of the modified oscillatory energy
errors against $t$.} \label{p3}
\end{figure}
 \begin{figure}[ptb]
\centering
\includegraphics[width=6cm,height=3cm]{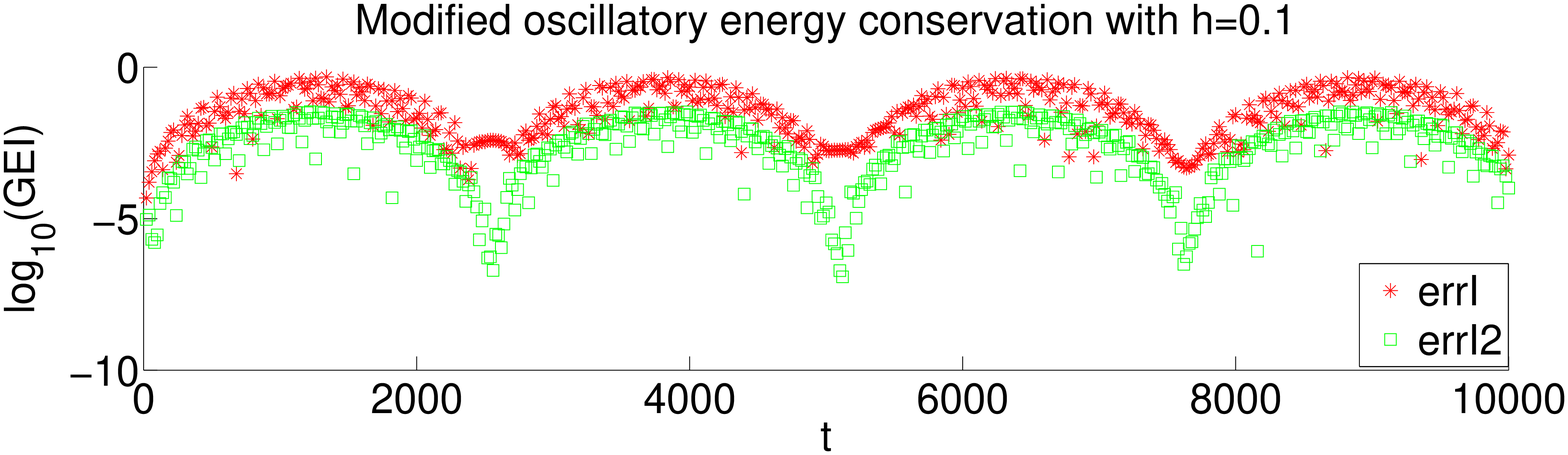}
\includegraphics[width=6cm,height=3cm]{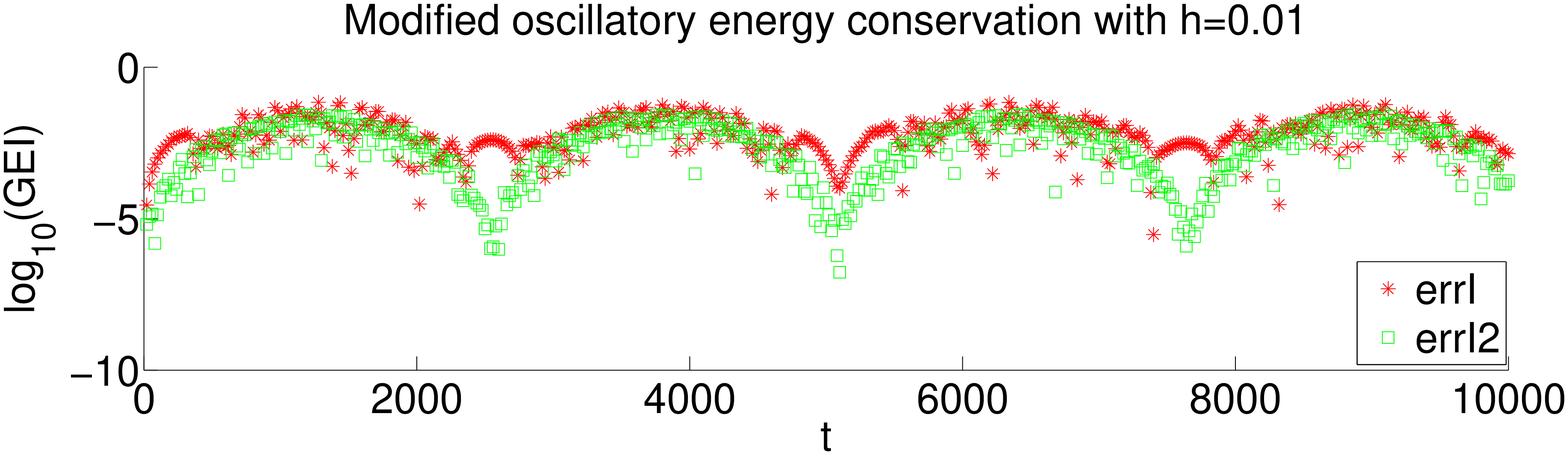}
\caption{AAVF 3: the logarithm of the modified oscillatory energy
errors against $t$.} \label{p4}
\end{figure}
\section{Conclusions} \label{sec:conclusions}
In this paper,  we   presented a long-term  analysis of the adapted
average vector field (AAVF) method for highly oscillatory
Hamiltonian systems.  This AAVF method can exactly preserve the
total energy of the underlying systems, but the main theme of this
paper is to study its oscillatory energy and the corresponding
numerical conservation. We analysed the long-term behaviour in the
oscillatory energy conservation  by developing modulated Fourier
expansions for the method. A further extension of the analysis to
multi-frequency case has also been discussed.

Last but not least, it is noted that some trigonometric
energy-preserving methods have been well developed for solving wave
equations and see
\cite{JMAA(2015)_Liu_Wu,JCP(2018)_Liu_Iserles_Wu,kai-2017,IMA2018}
for example. For Hamiltonian wave equations, the study of long-time
conservation of momentum and actions along these EP methods is
overarching importance, and this  will be discussed in our another
work.

\section*{Acknowledgements}

The authors are grateful to Professor Christian Lubich for his
helpful comments and discussions on the topic of modulated Fourier
expansions. We also thank him for drawing our attention to the
long-term  analysis of energy-preserving methods.


\end{document}